\documentclass[a4paper]{article}
\usepackage{latexsym}
\usepackage{amssymb}
\usepackage{amsthm}
\usepackage{amsmath,amsfonts}
\usepackage{multirow}
\usepackage{doi}

\usepackage[a4paper, margin=3cm]{geometry}

\theoremstyle{plain}
\newtheorem{theorem}{Theorem}
\newtheorem{lemma}[theorem]{Lemma}

\theoremstyle{definition}
\newtheorem{definition}[theorem]{Definition}

\newtheorem{conjecture}[theorem]{Conjecture}

\theoremstyle{remark}
\newtheorem{remark}[theorem]{Remark}

\def\titlerunning#1{\gdef\titrun{#1}}
\makeatletter
\def\author#1{\gdef\autrun{\def\and{\unskip, }#1}\gdef\@author{#1}}
\def\address#1{{%
\hypersetup{pdfborder={0 0 0}}%
\def\and{\\\hspace*{18pt}}%
\renewcommand{\thefootnote}{}%
\footnote{#1}}%
\markboth{\autrun}{\titrun}}
\makeatother
\def\email#1{\hspace*{4pt}{\em e-mail}: #1}

\titlerunning{}

\title{On an infinite sequence of strongly regular digraphs with parameters $(9(2n+3), 3(2n+3), 2n+4, 2n+1, 2n+4)$}
\author{Viktor A. Byzov and Igor A. Pushkarev}
\date{March, 2026}

\begin{document}
\maketitle

\address{Viktor A. Byzov, Igor A. Pushkarev: Department of Applied Mathematics and Informatics, Vyatka State University, Kirov, Russia;
\email{vbyzov@yandex.ru, god\_sha@mail.ru}
}

\begin{abstract}
The paper constructs an infinite sequence of strongly regular directed graphs. The construction is based on representing adjacency matrices as block matrices composed of circulant blocks, together with the use of a compactification operation consistent with polynomial arithmetic modulo 
$x^{2n+3}-1$. Using computer search with the pychoco library and subsequent analysis of automorphism groups in the GAP system, a stable structural pattern was identified, which made it possible to formulate and prove an explicit formula for the adjacency matrices of the infinite sequence of directed graphs. Among the obtained digraphs, there are examples with parameters 
$(63, 21, 8, 5, 8)$ and $(81, 27, 10, 7, 10)$, for which the question of existence had previously remained open. A hypothesis on the structure of the automorphism groups of the digraphs in the constructed sequence is also formulated.
\end{abstract}

\bigskip



\section{Introduction}

This paper deals with directed graphs (hereafter referred to as digraphs) without loops or multiple arcs in the same direction. The adjacency matrix of such digraphs is a square matrix with rows and columns indexed by the integers from 1 to $n$, where $n$ is the number of vertices; the entry at the intersection of the $i$-th row and $j$-th column is 1 if and only if there is an arc from vertex $i$ to vertex $j$. All other entries of the adjacency matrix are zero.

We describe an infinite family of strongly regular directed graphs with parameters $(9(2n+3), 3(2n+3), 2n+4, 2n+1, 2n+4)$ for integer $n\geqslant 1$. To construct the digraphs in this sequence, we use an approach based on representing the adjacency matrix in block-circulant form. This idea, borrowed from the work~\cite{Gritsenko}, was also applied by the authors in the work~\cite{Byzov2024} to find strongly regular directed graphs with parameters $(22, 9, 6, 3, 4)$. However, since the target digraphs in the present article have significantly more vertices, several additional constraints are imposed on the adjacency matrix to narrow the search space.

The authors developed a program that, using the constraint programming library pychoco, finds adjacency matrices of $\text{dsrg}(9(2n+3), 3(2n+3), 2n+4, 2n+1, 2n+4)$ for $n=1,\ldots,5$. For the obtained digraphs, the automorphism groups were found using the GAP system (see~\cite{GAP2026}). The digraphs constructed computationally made it possible to formulate and prove a theorem on an infinite sequence of strongly regular digraphs.

The paper is organized as follows. Section~2 introduces the necessary concepts and definitions and presents auxiliary statements. Section~3 describes the method of computer-assisted search for strongly regular digraphs. Section~4 proves the theorem on an infinite sequence of strongly regular digraphs.

\section{Preliminaries}
\label{sec:preliminary}

The following notation for standard matrices will be used throughout the text. $I_v$ denotes the identity matrix of order $v$, and $J_v$ denotes the square matrix of order $v$ with all entries equal to one.

The notion of a strongly regular digraph was first introduced by A.\,M.~Duval in~\cite{Duval1988} as a natural generalization of strongly regular undirected graphs. We now give two equivalent definitions of such digraphs.

\begin{definition}
\label{def:main1}
A strongly regular digraph with parameter set $(v, k, t, \lambda, \mu)$ is a directed graph on $v$ vertices satisfying the following conditions:
\begin{enumerate}
\item The outdegree and indegree of each vertex are both equal to $k$;
\item For each vertex $x$, there are exactly $t$ paths of the form $x \to z \to x$;
\item If there is a directed edge from $x$ to $y \ne x$, then there exist exactly $\lambda$ paths of the form $x \to z \to y$;
\item If there is no directed edge from $x$ to $y \ne x$, then there exist exactly $\mu$ paths of the form $x \to z \to y$.
\end{enumerate}
\end{definition}

\begin{definition}
\label{def:main2}
A strongly regular digraph with parameter set $(v, k, t, \lambda, \mu)$ is a directed graph on $v$ vertices whose adjacency matrix $A$ satisfies the following conditions:
\begin{gather*}
A^2 = tI_v + \lambda A + \mu(J_v - I_v - A),\\
A J_v = J_v A = k J_v.
\end{gather*}
\end{definition}

Instead of the phrase ``strongly regular digraph with parameter set $(v, k, t, \lambda, \mu)$'', we will often use the following shorter notation in the text: $\text{dsrg}(v, k, t, \lambda, \mu)$.

A circulant matrix (or circulant) is defined as a square matrix in which each row, starting from the second, is obtained by a cyclic right shift of the preceding row by one position. Denote by $\text{Circ}_{\mathbb{Z}}(n)$ the ring of circulant matrices of order $n$ with entries from~$\mathbb{Z}$. It is known that the ring $\text{Circ}_{\mathbb{Z}}(n)$ is isomorphic to the quotient ring \mbox{$\mathbb{Z}[x] \big/ (x^n - 1)$} consisting of polynomials of degree less than $n$ (see, for example,~\cite{Kra2012}). Under this isomorphism, the matrix
\begin{equation}
\begin{pmatrix}
a_1     & a_2 & \dots  & a_n \\
a_n     & a_1 & \dots  & a_{n-1} \\
\vdots  & \vdots  & \ddots & \vdots \\
a_2     & a_3 & \dots  & a_1 \\
\end{pmatrix}
\end{equation}
\noindent corresponds to the polynomial $a_1 + a_2x + \ldots + a_nx^{n-1}$.

Let $M$ be a block matrix whose blocks are all circulant matrices. The compactification of the matrix $M$ is defined as the matrix $M(x)$ obtained by replacing each block of $M$ with the corresponding polynomial under the isomorphism described above. As an example, consider one of the digraphs $\text{dsrg}(8, 3, 2, 1, 1)$, whose adjacency matrix can be represented as four circulant blocks of order four:
\begin{equation}
S = \left(
\begin{array}{cccc|cccc}
0 & 0 & 0 & 1 & 0 & 1 & 1 & 0 \\
1 & 0 & 0 & 0 & 0 & 0 & 1 & 1 \\
0 & 1 & 0 & 0 & 1 & 0 & 0 & 1 \\
0 & 0 & 1 & 0 & 1 & 1 & 0 & 0 \\ \hline
0 & 0 & 1 & 1 & 0 & 1 & 0 & 0 \\
1 & 0 & 0 & 1 & 0 & 0 & 1 & 0 \\
1 & 1 & 0 & 0 & 0 & 0 & 0 & 1 \\
0 & 1 & 1 & 0 & 1 & 0 & 0 & 0
\end{array}
\right).
\end{equation}

The matrix $S$ can be compactified to the matrix
\begin{equation}
S(x) =
\begin{pmatrix}
    x^3  & x+x^2\\
    x^2+x^3 & x
\end{pmatrix}.
\end{equation}

The inverse operation, which transforms a compactified matrix into a binary matrix consisting of circulant blocks, will be called decompactification.

A key property of compactification is that this operation is compatible with matrix addition and multiplication. Specifically, let $M_1$ and $M_2$ be two square block matrices consisting of circulant blocks of order $n$, and let $M_3 = M_1 + M_2$, $M_4 = M_1 \cdot M_2$. Then $M_3(x) \equiv M_1(x) + M_2(x) \pmod{x^n-1}$, $M_4(x) \equiv M_1(x) \cdot M_2(x) \pmod{x^n-1}$, where $M_1(x)$, $M_2(x)$, $M_3(x)$, and $M_4(x)$ are the compactifications of the corresponding matrices.

We use the notation $M[i,j]$ for the entry of matrix $M$ located at the intersection of the $i$-th row and $j$-th column, where rows and columns are numbered starting from one.

\section{A computer experiment to find strongly regular digraphs}
\label{sec:program}

We seek the adjacency matrices $A_n$ of $\text{dsrg}(9(2n+3), 3(2n+3), 2n+4, 2n+1, 2n+4)$ in the form of $9 \times 9$ block matrices, where each block is a circulant matrix of order $2n+3$. From Definition~\ref{def:main2}, it follows that the matrix $A_n$ satisfies the conditions
\begin{gather}
A_n^2 + 3A_n = (2n+4)J_{9(2n+3)}, \\
A_n J_{9(2n+3)} = J_{9(2n+3)} A_n = 3(2n+3) J_{9(2n+3)}.
\end{gather}

Denote by $A_n(x)$ the compactification of the matrix $A_n$, and by $Q_n(x) = \sum\limits_{i=0}^{2n+2}x^i$ the polynomial corresponding to the all-ones circulant block. Then the matrix $A_n(x)$ satisfies the following congruences:
\begin{gather}
A_n(x)^2 + 3A_n(x) \equiv (2n+4)J_9 Q_n(x) \pmod{x^{2n+3}-1}, \label{eq:equiv1}\\
A_n(x) J_9 Q_n(x) \equiv J_9 A_n(x) Q_n(x) \equiv 3(2n+3)J_9 Q_n(x) \pmod{x^{2n+3}-1}. \label{eq:equiv2}
\end{gather}

Since $x=1$ is a root of the polynomial $x^{2n+3}-1$, substituting into congruences~\eqref{eq:equiv1} and~\eqref{eq:equiv2} yields the equalities
\begin{gather}
A_n(1)^2+3A_n(1) = (2n+3)(2n+4)J_9, \label{eq:An1_1}\\
A_n(1)J_9 = J_9 A_n(1) = 3(2n+3)J_9. \label{eq:An1_2}
\end{gather}

The following lemma describes a sequence of matrices that satisfy equations~\eqref{eq:An1_1} and~\eqref{eq:An1_2}.

\begin{lemma}
\label{lemma:Cn}
The sequence of matrices
\begin{equation}
C_n = \left(
\begin{array}{ccccccccc}
0 & n+1 & n+1 & n+1 & n+1 & 1 & 2 & n+1 & n+1 \\
0 & n+1 & n+1 & n+1 & n+1 & 1 & 2 & n+1 & n+1 \\
0 & n+1 & n+1 & n+1 & n+1 & 1 & 2 & n+1 & n+1 \\
0 & n+1 & n+1 & n+1 & n+1 & 1 & 2 & n+1 & n+1 \\
0 & n+1 & n+1 & n+1 & n+1 & 1 & 2 & n+1 & n+1 \\
0 & n+1 & n+1 & n+1 & n+1 & 1 & 2 & n+1 & n+1 \\
2n+3 & 1 & 1 & 1 & 1 & 2n+1 & 2n-1 & 1 & 1 \\
2n+3 & 1 & 1 & 1 & 1 & 2n+1 & 2n-1 & 1 & 1 \\
2n+3 & 1 & 1 & 1 & 1 & 2n+1 & 2n-1 & 1 & 1
\end{array}
\right)
\end{equation}
has the following properties:
\begin{gather}
C_n^2+3C_n = (2n+3)(2n+4)J_9, \label{eq:Cn_1}\\
C_n J_9 = J_9 C_n = 3(2n+3)J_9. \label{eq:Cn_2}
\end{gather}
\end{lemma}

The validity of this lemma can be verified by direct substitution of the formula for $C_n$ into equations~\eqref{eq:Cn_1} and~\eqref{eq:Cn_2}. The verification can be performed by direct computation or using a computer algebra system.

Thus, the sequence of matrices $C_n$ can be regarded as a ``potential candidate'' for the role of $A_n(1)$.

\begin{remark}
Note that a suitable sequence of matrices $C_n$ was obtained through nonsystematic computer experiments using the constraint programming library pychoco (see~\cite{pychoco2025}).
\end{remark}

We will search specifically for the matrices $A_n(x)$, since the adjacency matrices of the digraphs can be uniquely recovered from them. Exhaustive enumeration of all possible $A_n(x)$ is infeasible even for small $n$, so we impose the following additional constraints to reduce the search space.

\begin{enumerate}
\item $A_n(1) = C_n$.
\item $A_n(x)[1, 2] = 1+x+x^2+\ldots+x^n$.
\item $A_n(x)[7, 2] = 1$.
\item For $1\leqslant j \leqslant 9$,
\begin{equation}
A_n(x)[7, j] = A_n(x)[8, j] = A_n(x)[9, j].
\end{equation}
\item For $2 \leqslant i \leqslant 6$, $1\leqslant j \leqslant 9$,
\begin{equation}
A_n(x)[i, j] \equiv x\cdot A_n(x)[i-1, j] \pmod{x^{2n+3}-1}.
\end{equation}
\item For $1 \leqslant i \leqslant 9$,
\begin{equation}
A_n(x)[i, 3] \equiv x\cdot A_n(x)[i, 2] \pmod{x^{2n+3}-1}.
\end{equation}
\item For $1 \leqslant i \leqslant 9$,
\begin{equation}
A_n(x)[i, 8] \equiv x\cdot A_n(x)[i, 9] \pmod{x^{2n+3}-1}.
\end{equation}
\end{enumerate}

\begin{remark}
The described conditions were ``guessed'' through numerous computer experiments: with this set of conditions, the software search described below yielded positive results.
\end{remark}

A program was written to search for matrices $A_n(x)$ satisfying conditions (1)--(7). The constraint programming library pychoco was used. The program found all suitable matrices for $n=1,\ldots, 5$. For the found digraphs, the automorphism groups were computed using the GAP system. The results are presented in Table~\ref{tab:dsrg_search}. The third column of this table gives the number of nonisomorphic digraphs among those found.

\begin{table}[ht!]
\centering
\caption{Search results for strongly regular digraphs}
\label{tab:dsrg_search}
\begin{tabular}{|c|c|c|c|c|}
\hline
\textbf{$n$} & \textbf{Parameters} & \begin{tabular}{@{}c@{}} \textbf{Search} \\ \textbf{time} \\ \textbf{(sec.)} \\ \end{tabular} & \begin{tabular}{@{}c@{}} \textbf{Number of} \\ \textbf{dsrgs} \end{tabular} & \begin{tabular}{@{}c@{}} \textbf{Automorphism } \\ \textbf{groups} \end{tabular} \\ \hline
\rule{0pt}{2.6ex} \multirow{3}{*}{1} & \multirow{3}{*}{$(45,15,6,3,6)$} & \multirow{3}{*}{13.85} & 1 & $C_2 \times (C_2^4 \rtimes C_5)$ \\ \cline{4-5}
\rule{0pt}{2.6ex} & & & 3 & $C_2^2 \times (C_2^8 \rtimes C_5)$ \\ \cline{4-5}
\rule{0pt}{2.6ex} & & & 2 & $C_2 \times (C_3^5 \rtimes (C_2 \times (C_2^8 \rtimes C_5)))$ \\ \hline
\rule{0pt}{2.6ex} 2 & $(63,21,8,5,8)$ & 62.10 & 4 & $C_2 \times (C_2^6 \rtimes C_7)$ \\ \hline
\rule{0pt}{2.6ex} 3 & $(81,27,10,7,10)$ & 3388 & 2 & $C_2 \times (C_2^8 \rtimes C_9)$ \\ \hline
\rule{0pt}{2.6ex} 4 & $(99,33,12,9,12)$ & 1703 & 2 & $C_2 \times (C_2^{10} \rtimes C_{11})$ \\ \hline
\rule{0pt}{2.6ex} 5 & $(117,39,14,11,14)$ & 90354 & 2 & $C_2 \times (C_2^{12} \rtimes C_{13})$ \\ \hline
\end{tabular}
\end{table}

The search was performed on a computer with an Intel Core i5-7400 processor (3.0~GHz) and 32 GB of RAM. All found digraphs are available in the repository~\cite{Byzov2026}. Note that, at the time of writing, there was no information in Table~\cite{Brouwer2026} regarding the existence of digraphs with parameters $(63, 21, 8, 5, 8)$ and $(81, 27, 10, 7, 10)$.

\section{An infinite sequence of strongly regular digraphs}
\label{sec:sequence}

It was observed that several of the found digraphs have similar structures. This led to a conjecture whose proof yielded an infinite sequence of strongly regular digraphs.

Introduce the following notation. Let $P(x)=\sum\limits_{i=0}^{n}x^i$, $Q(x)=\sum\limits_{i=0}^{2n+2}x^i$,\\$R(x)=Q(x)-x^{n-1}-x^{2n+1}$, $S(x)=Q(x) - 1 - x^2 - x^{n+2} - x^{n+3}$.

We now prove some useful properties of these polynomials.
\begin{lemma}
\label{lemma:PQ}
The following statements hold (all congruences are modulo $x^{2n+3} - 1$):
\begin{enumerate}
\item $x^sQ(x) \equiv Q(x)$ for any integer $s\geqslant 0$;
\item $P(x)\cdot Q(x) \equiv (n+1)Q(x)$;
\item $Q(x)^2 \equiv (2n + 3)Q(x)$;
\item $(1-x)P(x) = 1-x^{n+1}$, $(1-x)Q(x)=1-x^{2n+3}$.
\end{enumerate}
\end{lemma}
\begin{proof}
The validity of item (1) follows, for example, from the fact that multiplication by $x^s$ simply cyclically shifts the exponents, which leaves $Q(x)$ unchanged since this polynomial contains all possible exponents from 0 to $2n+2$.

The validity of items (2) and (3) follows from item (1) and the decompositions $P(x) \cdot Q(x) = \sum_{i=0}^{n} x^i Q(x)$ and $Q(x)^2 = \sum_{i=0}^{2n+2} x^i Q(x)$.

The last item of the lemma is verified by direct computation. The lemma is proved.
\end{proof}

Consider the sequence of matrices $A_n(x)$ for $n \geqslant 2$, defined by the formula
{\small\begin{equation}
\label{eq:seq_matrix}
\left(
\begin{array}{ccccccccc}
0 & P & xP & x^{n-1}P & x^{n-2}P & x^{2n} & x(1+x^n) & x^2P & xP \\
0 & xP & x^2P & x^nP & x^{n-1}P & x^{2n+1} & x^2(1+x^n) & x^3P & x^2P \\
0 & x^2P & x^3P & x^{n+1}P & x^nP & x^{2n+2} & x^3(1+x^n) & x^4P & x^3P \\
0 & x^3P & x^4P & x^{n+2}P & x^{n+1}P & 1 & x^4(1+x^n) & x^5P & x^4P \\
0 & x^4P & x^5P & x^{n+3}P & x^{n+2}P & x & x^5(1+x^n) & x^6P & x^5P \\
0 & x^5P & x^6P & x^{n+4}P & x^{n+3}P & x^2 & x^6(1+x^n) & x^7P & x^6P \\
Q & 1 & x & x^{n-1} & x^{n-2} & R & S & x^2 & x \\
Q & 1 & x & x^{n-1} & x^{n-2} & R & S & x^2 & x \\
Q & 1 & x & x^{n-1} & x^{n-2} & R & S & x^2 & x
\end{array}
\right),
\end{equation}}
\noindent where all operations are performed in the ring $\mathbb{Z}[x] \big/ (x^{2n+3} - 1)$. For brevity, $P(x)$, $Q(x)$, $R(x)$, and $S(x)$ are denoted simply by $P$, $Q$, $R$, and $S$, respectively.

\begin{theorem}
\label{th:main}
The digraph whose adjacency matrix coincides with the decompactification of matrix~\eqref{eq:seq_matrix} for $n \geqslant 2$ is $\text{dsrg}(9(2n+3), 3(2n+3), 2n+4, 2n+1, 2n+4)$.
\end{theorem}
\begin{proof}
To prove the theorem, it suffices to show that the sequence of matrices $A_n(x)$ for $n \geqslant 2$ satisfies the regularity condition and congruence~\eqref{eq:equiv1}. The fact that $A_n(x)$ satisfies the regularity condition follows from Lemma~\ref{lemma:Cn} and $A_n(1) = C_n$.

We now prove that $A_n(x)^2 + 3A_n(x) \equiv (2n+4)J_9 Q(x)$. Denote $A_n(x)^2 + 3A_n(x)$ by~$W(x)$. Since in the first six rows of $A_n(x)$, each row starting from the second is obtained by multiplying the previous row by $x$, it suffices to compute only the first row of $W(x)$. Similarly, due to the equality of the last three rows of $A_n(x)$, we need to compute only the seventh row instead of the last three rows of $W(x)$. In the computations, we will use the properties from Lemma~\ref{lemma:PQ}.
\begin{multline}
W(x)[1, 1] = (x^{2}P(x) + x P(x) + x^{n + 1} + x) Q(x) \equiv\\\equiv (n+1)Q(x)+(n+1)Q(x)+2Q(x) = (2n+4) Q(x).
\end{multline}

We now proceed to compute the next entry of the matrix $W(x)$:
\begin{equation}
W(x)[1, 2] = (2x^{n+2} + x^3 + x)P(x)^2+(x^{2n+5}+x^2+x+3)P(x)+x^{n+1}+x.
\end{equation}

Using Lemma~\ref{lemma:PQ} and the relation $x^{2n+3} \equiv 1$, it can be shown that
\begin{equation}
(1-x)W(x)[1, 2] \equiv 0 \pmod{x^{2n+3}-1}.
\end{equation}

Hence, there exists a polynomial $T(x)$ in the ring $\mathbb{Z}[x]$ such that
\begin{equation}
\label{eq:Wx}
(1-x)W(x)[1, 2] = T(x)(1-x^{2n+3}) = T(x)(1-x)Q(x).
\end{equation}

Since $\mathbb{Z}[x]$ is an integral domain and $1 - x \neq 0$, equation~\eqref{eq:Wx} implies that
\begin{equation}
W(x)[1, 2] = T(x)Q(x).
\end{equation}

Since $\deg W(x)[1, 2] \leqslant 2n+2$ and $\deg Q(x) = 2n+2$, the polynomial $T(x)$ must be a constant. As $W(1)[1, 2] = 2(2n+3)(n+2)$ and $Q(1) = 2n+3$, it follows that $T(x) = 2n+4$. Thus, $W(x)[1, 2] = (2n+4) Q(x)$.

Furthermore, $W(x)[1, 3] = x\cdot W(x)[1, 2] \equiv (2n+4)Q(x)$.

$W(x)[1, 4] = x^{n-1}\cdot W(x)[1, 2] \equiv (2n+4)Q(x)$.

$W(x)[1, 5] = x^{n-2}\cdot W(x)[1, 2] \equiv (2n+4)Q(x)$.

After simplification, we obtain
\begin{multline}
W(x)[1, 6] \equiv (2n+4)Q(x)+(1-x)P(x)(x^{2n+1}+x^n+2x^{n-1})+\\+2x^{2n}-x^n-x^{n-1}.
\end{multline}

Item (4) of Lemma~\ref{lemma:PQ} allows us to simplify: $W(x)[1, 6] \equiv (2n+4) Q(x)$.

Using the expression for $S(x)$, we obtain
\begin{multline}
W(x)[1, 7] \equiv\\\equiv (2n+4)Q(x)+(1-x)P(x)(x^{n+4}+2x^{n+3}+x^{n+2}+x^2+x+2)-\\-x^{n+4}-x^{n+3}+2x^{n+1}+x-1.
\end{multline}

Item (4) of Lemma~\ref{lemma:PQ} allows the simplification: $W(x)[1, 7] \equiv (2n+4) Q(x)$.

Furthermore, $W(x)[1, 8] = x^2\cdot W(x)[1, 2] \equiv (2n+4)Q(x)$.

$W(x)[1, 9] = x\cdot W(x)[1, 2] \equiv (2n+4)Q(x)$.
\begin{multline}
W(x)[7, 1] = (S(x)+x^2+x+3)Q(x) =\\= (Q(x)-1-x^{n+2}-x^{n+3}+x+3)Q(x)\equiv (2n+4)Q(x).
\end{multline}

Using the statements of Lemma~\ref{lemma:PQ}, we obtain
\begin{equation}
W(x)[7, 2] \equiv (2n + 3) Q(x) + x^{n+2}P(x)+xP(x)+1.
\end{equation}

From the definitions of $P(x)$ and $Q(x)$, it follows that $x^{n+2} P(x) + x P(x) + 1 = Q(x)$, so $W(x)[7, 2] \equiv (2n + 4) Q(x)$.

Furthermore, $W(x)[7, 3] = x\cdot W(x)[7, 2] \equiv (2n+4)Q(x)$.

$W(x)[7, 4] = x^{n-1}\cdot W(x)[7, 2] \equiv (2n+4)Q(x)$.

$W(x)[7, 5] = x^{n-2}\cdot W(x)[7, 2] \equiv (2n+4)Q(x)$.

Using items (1) and (3) of Lemma~\ref{lemma:PQ} and a sequence of transformations, it is proved that $W(x)[7, 6] \equiv W(x)[7, 7] \equiv (2n+4) Q(x)$.

Furthermore, $W(x)[7, 8] = x^2\cdot W(x)[7, 2] \equiv (2n+4)Q(x)$.

$W(x)[7, 9] = x\cdot W(x)[7, 2] \equiv (2n+4)Q(x)$.

Thus, we have shown that all entries in the first and seventh rows of the matrix $W(x)$ equal $(2n+4) Q(x)$. Consequently, all other entries of this matrix also equal $(2n+4) Q(x)$. The theorem is proved.
\end{proof}

The authors developed a program to generate strongly regular digraphs based on the formula from Theorem~\ref{th:main}. For the generated digraphs, the automorphism groups were computed using GAP, which led to the following conjecture.

\begin{conjecture}
For the constructed sequence of strongly regular digraphs with parameters $(9(2n+3), 3(2n+3), 2n+4, 2n+1, 2n+4)$, the automorphism group is $C_2 \times (C_2^{2n+2} \rtimes C_{2n+3})$.
\end{conjecture}

\section{Conclusion}
\label{sec:conclusion}
In this paper, an explicit formula is obtained for a family of strongly regular digraphs with parameters $(9(2n+3), 3(2n+3), 2n+4, 2n+1, 2n+4)$. The adjacency matrix of each digraph consists of 81 circulant submatrices. A conjecture is formulated that the automorphism groups of the obtained digraphs are $C_2 \times (C_2^{2n+2} \rtimes C_{2n+3})$.


\begin{thebibliography}{99}

\bibitem{Gritsenko}
O.~Gritsenko On strongly regular graph with parameters (65; 32; 15; 16). 2021. 10 p. (Cornell Univ. Libr. e-Print Archive;
arXiv:2102.05432). \doi{10.48550/arXiv.2102.05432}.

\bibitem{Byzov2024}
V.\,A.~Byzov and I.\,A.~Pushkarev,
On the existence of directed strongly regular graphs with parameters
(22, 9, 6, 3, 4). \emph{Prikl. Diskr. Mat.}, 66: 86--96, 2024. \doi{10.17223/20710410/66/8}.

\bibitem{GAP2026}
The GAP~Group,
GAP -- Groups, Algorithms, and Programming, Version 4.15.1.
\url{https://www.gap-system.org}, 2026.

\bibitem{Duval1988}
A.\,M.~Duval,
A directed graph version of strongly regular graphs.
\emph{J. Combin. Theory Ser.~A}, 47(1): 71--100, 1988.
\doi{10.1016/0097-3165(88)90043-X}.

\bibitem{Kra2012}
I.~Kra and S.\,R.~Simanca,
On circulant matrices. \emph{Notices Amer. Math. Soc.}, 59(3): 368--377, 2012.
\doi{10.1090/noti804}.

\bibitem{pychoco2025}
D.~Justeau-Allaire and C.~Prud’homme,
pychoco: all-inclusive Python bindings for the Choco-solver constraint
programming library.
\emph{J. Open Source Softw.}, 10(113): 8847, 2025. \doi{10.21105/joss.08847}.

\bibitem{Byzov2026}
V.\,A.~Byzov,
Search results for directed strongly regular graphs with parameters
$(9(2n+3), 3(2n+3), 2n+4, 2n+1, 2n+4)$.
\url{https://github.com/byzovv/dsrg_search}, 2026.

\bibitem{Brouwer2026}
A.\,E.~Brouwer and S.\,A.~Hobart,
Parameters of directed strongly regular graphs. \url{https://homepages.cwi.nl/~aeb/math/dsrg/dsrg.html}, 2026.

\end{thebibliography}
\end{document}